\documentclass[a4paper,reqno]{amsart}

\usepackage{enumerate}
\usepackage{amssymb} 
\usepackage[colorlinks=true]{hyperref}
\hypersetup{urlcolor=blue, citecolor=cyan}
\hypersetup{citecolor=blue}

\numberwithin{equation}{section}

\newtheorem{thm}{Theorem}[section]
\newtheorem{lem}[thm]{Lemma}

\newtheorem{prop}[thm]{Proposition}
\theoremstyle{definition}
\newtheorem{rem}[thm]{Remark}

\newcommand\R{{\mathbb R}}
\newcommand\C{{\mathbb C}}
\newcommand\N{{\mathbb N}}

\newcommand\Cset{ K }
\newcommand\IL{ \mu }
\newcommand\Foru{{\mathcal P}}
\newcommand\Err{{\mathcal E}}

\newcommand\Calpha {C_\alpha }
\newcommand\Csu{ M }

\newcommand\Loc{{\mathrm{loc}}}

\newcommand\goto{\mathop{\longrightarrow}}

\newcommand\MScN[1]{\href{http://www.ams.org/mathscinet-getitem?mr=#1}{\nolinkurl{(#1)}}}
\newcommand\DOI[1]{\href{http://dx.doi.org/#1}{(doi: \nolinkurl{#1})}}
\newcommand\LINK[1]{\href{#1}{(link: \nolinkurl{#1})}}

\begin{document}

\title{Blowup on an arbitrary compact set for a Schr\"o\-din\-ger equation with nonlinear source term}

\def\shorttitle{Blowup on a compact set}

\author[T. Cazenave]{Thierry Cazenave$^1$}
\email{\href{mailto:thierry.cazenave@sorbonne-universite.fr}{thierry.cazenave@sorbonne-universite.fr}}

\author[Z. Han]{Zheng Han$^2$}
\email{\href{mailto:hanzh_0102@hznu.edu.cn}{hanzh\_0102@hznu.edu.cn}}

\author[Y. Martel]{Yvan Martel$^{3}$}
\email{\href{mailto:yvan.martel@polytechnique.edu}{yvan.martel@polytechnique.edu}}

\thanks{ZH thanks NSFC 11671353,11401153, Zhejiang Provincial Natural Science Foundation of China under Grant No. LY18A010025, and CSC for their financial support; and the Laboratoire Jacques-Louis Lions for its kind hospitality}

\address{$^1$Sorbonne Universit\'e \& CNRS, Laboratoire Jacques-Louis Lions,
B.C. 187, 4 place Jussieu, 75252 Paris Cedex 05, France}

\address{$^2$Department of Mathematics, Hangzhou Normal University, Hangzhou, 311121, China}

\address{$^3$CMLS, \'Ecole Polytechnique, CNRS, 91128 Palaiseau Cedex, France}

\subjclass[2010]{Primary 35Q55; Secondary 35B44, 35B40}

\keywords{Nonlinear Schr\"o\-din\-ger equation, finite-time blowup, blow-up set, blow-up profile}

\begin{abstract}
We consider the nonlinear Schr\"o\-din\-ger equation on $\R^N $, $N\ge 1$,
\[ 
\partial _t u = i \Delta u + \lambda  | u |^\alpha u \quad \mbox{on $\R^N $, $\alpha>0$,}
\]
with $\lambda \in \C$ and $\Re \lambda >0$,  for $H^1$-subcritical  nonlinearities, i.e. $\alpha >0$ and $(N-2) \alpha < 4$.
Given a compact set $\Cset \subset \R^N $, we construct $H^1$ solutions that are defined on $(-T,0)$ for some $T>0$, and blow up on $\Cset $ at $t=0$. 
The construction is based on an appropriate ansatz. The initial ansatz is simply $U_0(t,x) = ( \Re \lambda  )^{- \frac {1} {\alpha }}  (-\alpha t + A(x) )^{ -\frac {1} {\alpha } - i \frac {\Im \lambda } {\alpha \Re \lambda } }$, where $A\ge 0$ vanishes exactly on $ \Cset $, which is a solution of the ODE $u'= \lambda | u |^\alpha u$. 
We refine this ansatz inductively, using ODE techniques. 
We complete the proof by energy estimates and a compactness argument.
This strategy is reminiscent of~\cite{CMZ-NLS, CMZ-wave1}.
\end{abstract}

\maketitle


\section{Introduction}

We consider the nonlinear Schr\"o\-din\-ger equation
\begin{equation} \label{NLS1} 
\partial _t u = i \Delta u + \lambda  | u |^\alpha u
\end{equation} 
on $\R^N $, where 
\begin{equation} \label{fDFG0} 
\alpha >0, \quad 
(N-2) \alpha <  4. 
\end{equation}
and $\Re \lambda >0$. Note that by scaling invariance of~\eqref{NLS1}, we may assume that $\Re \lambda =1$, so we write
\begin{equation} \label{fCL1} 
 \lambda =1 + i \IL , \quad \IL \in \R.
\end{equation} 
Under assumption~\eqref{fDFG0}, equation~\eqref{NLS1} is $H^1$-subcritical, so that the corresponding Cauchy problem is locally well posed in $H^1 (\R^N )$.

Concerning blowup, it is proved in~\cite[Theorem~1.1]{CazenaveCDW-Fuj} that, for $ \alpha <2/N$, equation~\eqref{NLS1} has no global in time $H^1$ solution that remains bounded in $H^1$. In other words, every $H^1$ solution blows up, in finite or infinite time.
Moreover, it is proved in~\cite{CMZ-NLS} that under assumption~\eqref{fDFG0} with the restriction $\alpha \ge 2$, and for $\lambda =1$, finite-time blowup occurs. 
This result is extended in~\cite{KawakamiM} to the case $\alpha >1$ and $\lambda \in \C$ with $(\alpha +2) \Re \lambda \ge \alpha  |\lambda |$. 

In this paper, we extend the previous blow-up results to the whole range of $H^1$ subcritical powers and arbitrary $\lambda \in \C$ with $\Re \lambda >0$. Moreover, we prove blowup on any prescribed compact subset of $\R^N $. Our result is the following.

\begin{thm} \label{eThm1} 
Let $N\ge 1$, let $\alpha >0$ satisfy~\eqref{fDFG0}, let $\IL \in \R$ and $\lambda $ given by~\eqref{fCL1}, and let $\Cset $ be a nonempty compact subset of $\R^N $.
It follows that there exist $T>0$ and a solution $u\in C((-T , 0), H^1 (\R^N ) ) $ of~\eqref{NLS1} 
 which blows up at time $0$ exactly on $ \Cset $ in the following sense.
\begin{enumerate}[{\rm (i)}] 

\item \label{eThm1:1}  If $x_0\in  \Cset $ then for any $r>0$,
\begin{equation}\label{xE}
\lim_{t\uparrow 0} \|u(t)\|_{L^2(|x-x_0|<r)}=\infty .
\end{equation}

 \item \label{eThm1:2} 
If $U $ is an open subset of $\R^N $ such that $ \Cset \subset U$, then 
\begin{equation}\label{fEssupg}
\lim_{t\uparrow 0} \| \nabla u(t)\|_{L^2( U )}=\infty .
\end{equation}

\item \label{eThm1:3}  If $\Omega $ is an open subset of $\R^N $ such that $ \overline{\Omega } \cap \Cset = \emptyset$, then
 \begin{equation}\label{xnotE}
 \sup_{t\in( -T, 0] }   \|u(t)\|_{ H^1 ( \Omega  )} <\infty .
\end{equation}

\end{enumerate}
\end{thm} 

\begin{rem} \label{eRM1} 
Here are some comments on Theorem~\ref{eThm1}.
\begin{enumerate}[{\rm (i)}] 

\item \label{eRM1:1} 
Estimate~\eqref{xE} can be refined. More precisely, it follows from~\eqref{fEstLp:2} that $ (- t)^{-\frac 1{\alpha }+\frac N{2k}} \lesssim  \| u (t)\|_{L^2 ( |x-x_0| <r )} \lesssim   (- t)^{-\frac 1{\alpha }}$ where $k > \frac {N\alpha } {2}$ is given by~\eqref{feApprox1:3}--\eqref{feApprox1:7}.  

\item \label{eRM1:2}
From the proof of Theorem~\ref{eThm1}, the blow-up mechanism for $u$ is described as follows. 
If $U_0 (t, x) = (- \alpha t + A (x))^{- \frac {1 } {\alpha } - i \frac {\IL} {\alpha }}$ 
where $A\ge 0$ (which vanishes exactly on $K$) is given by~\eqref{fDfnA}, then
$u (t) = U_0 (t) + V (t) + \varepsilon (t) $, where $ \| V (t) \| _{ H^1 }\lesssim (-t)^\nu  \| U_0 (t)\| _{ L^2 }$ and $ \| \varepsilon (t) \| _{ H^1 } \lesssim (-t)^\nu $ for some $\nu >0$, see~\eqref{feEstUJ:3} and~\eqref{fEEN24}.

\end{enumerate} 
\end{rem} 

To prove Theorem~\ref{eThm1}, we follow the strategy, introduced in~\cite{CMZ-NLS} (see also the references there), of defining the ansatz $U_0$,  blowing-up solution of the ODE $\partial _t U_0 = \lambda  |U_0|^\alpha U_0$, and then using energy estimates and compactness arguments. 
In~\cite{CMZ-NLS}, restricted to $\alpha \ge 2$ and $\lambda =1$, $U_0$ is a sufficiently good approximation and blowup is proved at any finite number of points.
To treat any subcritical $\alpha $ and any $\lambda \in \C$ with $\Re \lambda >0$, we need to refine the ansatz following the technique developed in~\cite{CMZ-wave1} for the semilinear wave equation. 
We emphasize that this technique only uses ODE arguments. 
See the beginning of Sections~\ref{sBUAnsatz} and~\ref{sApprox} for more details. 
See also Remark~\ref{eCRem1} below for comments on the restriction~\eqref{fDFG0}  to $H^1$-subcritical powers.

The rest of this paper is organized as follows. In Section~\ref{sNEE} we introduce some notation that we use throughout the paper and we recall some useful estimates. In Section~\ref{sBUAnsatz} we construct the appropriate blow-up ansatz. Section~\ref{sApprox} is devoted to the construction of a sequence of solutions of~\eqref{NLS1} close to the blow-up ansatz and to estimates of this sequence. Finally, we complete the proof of Theorem~\ref{eThm1} in Section~\ref{sProof}. 

\section{Notation and preliminary estimates} \label{sNEE} 

\subsection{Some Taylor's inequalities}
Let 
\begin{equation} \label{fNEE1} 
f(u)=  |u|^\alpha  u .
\end{equation} 
In general, $f$ is not $C^1$ as a function $\C\to \C$ (except for $\alpha \in 2\N$, when $f$ is analytic).
However, $f$ is $C^1$ as a function $\R^2 \to \R^2$. We denote by $df$ the derivative of $f$ is this sense, and we have
\begin{equation} \label{fDerf2} 
df (u) v = \frac {\alpha +2} {2}  |u|^\alpha v + \frac {\alpha } {2}  |u|^{\alpha -2} u^2  \overline{v} .
\end{equation} 
We also have
\begin{equation*} 
\nabla f(u) = \frac {\alpha +2} {2}  |u|^\alpha  \nabla u + \frac {\alpha } {2}  |u|^{\alpha -2} u^2 \nabla  \overline{u} ,
\end{equation*} 
so that
\begin{equation} \label{fDerf2:b1} 
\Re (\nabla (f(u))  \cdot \nabla  \overline{u} )= \frac {\alpha +2} {2}  |u|^\alpha   |\nabla u|^2 + \frac {\alpha } {2} \Re ( |u|^{\alpha -2} u^2 (\nabla  \overline{u})^2) \ge   |u|^\alpha   |\nabla u|^2. 
\end{equation} 
In addition, we have the following estimates.

\begin{lem} \label{eANL1} 
Set 
\begin{equation} \label{fDerf3} 
\Calpha  = 
\begin{cases} 
0 & 0<\alpha \le 1 \\ 1 & \alpha >1.
\end{cases} 
\end{equation} 
There exists a constant $\Csu \ge 1$ such that
\begin{gather} 
 |df  (u)| \le \Csu  |u|^\alpha  \label{feANL1:1} \\
  | f  (u+v) - f  (u) | \le \Csu ( |u|^\alpha +  |v|^\alpha )  |v|  \label{feANL1:2} \\
 |df  (u+v) - df  (u) | \le \Csu (  |v|^\alpha  + \Calpha   |u|^{\alpha -1}  |v| ) \label{fDerf4:1:b1}  \\
|df  (u+v) - df  (u) | \le \Csu (  |u|^\alpha  +  |v|^\alpha  ) \label{feANL1:3} \\
 |f  (u+v) - f  (u) - df   (u) v| \le \Csu (  |v|^{\alpha +1} + \Calpha   |u|^{\alpha -1}  |v|^2) \label{fDerf5} 
\end{gather} 
for all $u,v\in \C$. Moreover, 
\begin{equation} \label{fDerf6} 
 | \nabla [f  (u+v) - f  (u) - f  (v)]  | \le \Csu ( |u|^\alpha +  |v|^\alpha ) ( |\nabla u| +  |\nabla v|) 
\end{equation} 
a.e. for all $u, v\in H^1 (\R^N ) $.
\end{lem} 

\begin{proof} 
Estimate~\eqref{feANL1:1} is an immediate consequence of~\eqref{fDerf2}, and~\eqref{feANL1:2} follows, using 
\begin{equation*} 
f  (u+v)- f (u) =  \int _0^1 \frac {d} {dt} f  (u+tv) \,dt =  \int _0^1 df  (u+tv) v \,dt.
\end{equation*} 
Estimate~\eqref{fDerf4:1:b1} is classical, see e.g.~\cite[formulas~(2.26)-(2.27)]{CFH}; and~\eqref{feANL1:3} follows from~\eqref{fDerf4:1:b1} and the elementary estimate
\begin{equation*} 
|v|^\alpha  + \Calpha  |u|^{\alpha -1}  |v| \le 2(  |u|^\alpha +  |v|^\alpha ).
\end{equation*} 
Writing
\begin{equation*} 
f  (u+v)- f  (u) - df  (u)v = \int _0^1 [ df  (u+\theta v)v - df  (u)v ]\, d\theta 
\end{equation*}  
and using~\eqref{fDerf4:1:b1}, we obtain~\eqref{fDerf5}. 
Finally, given $u,v\in H^1 (\R^N ) $,
\begin{equation*} 
\begin{split} 
\partial  _{ x_j }[f  (u+v)- f  (u)- f(v)] & \\ = [df  (u+v) - & df  (u) ]\partial  _{ x_j } u + [df  (u+v) - df   (v) ]\partial  _{ x_j } v
\end{split} 
\end{equation*} 
so that~\eqref{fDerf6}  follows from~\eqref{feANL1:3}.
\end{proof}

\subsection{A Sobolev inequality}
We have
\begin{equation} \label{fSob1} 
 | u |^\alpha  |\nabla u|^2 \ge  | u |^\alpha  |\nabla  |u| |^2 = \frac {4} {(\alpha +2)^2}  |\nabla (  |u|^{\frac {\alpha +2} {2}} )|^2.
\end{equation} 
Note that by~\eqref{fDFG0} 
\begin{equation*} 
\theta : = \frac { N \alpha  } {4(\alpha +1)} \in (0,1) .
\end{equation*}
Since $ \frac {\alpha +2} {4( \alpha +1)} = \frac {1} {2} -  \frac {\theta } {N}$ ,
we deduce from Gagliardo-Nirenberg's inequality that 
\begin{equation*} 
\| v \| _{L^{\frac {4(\alpha +1) } {\alpha +2}}}\lesssim  \| v \| _{L^2}^{1- \theta }  \| \nabla v \| _{ L^2 } ^{\theta },
\end{equation*}
so that, letting $v =   |u|^{\frac {\alpha +2} {2}} $ and using~\eqref{fSob1}, 
\begin{equation*} 
\int  |u|^{2\alpha +2} \lesssim   \Bigl( \int  |u|^{\alpha +2} \Bigr)^{  \frac {4 - (N-4) \alpha  } { 2 (\alpha + 2) }}  \Bigl( \int  |u|^\alpha   | \nabla u |^2 \Bigr)^ {\frac { N\alpha  } {2 (\alpha + 2) }} .
\end{equation*} 
Since $ \frac { N\alpha  } {2 (\alpha + 2) } <1$ by~\eqref{fDFG0}, we see that for every $\eta >0$, there exists $C_\eta>0$ such that
\begin{equation} \label{fDerf7} 
\int  |u|^{2\alpha +2} \le \eta \int  |u|^\alpha   | \nabla u |^2 + C_\eta \int  |u|^{\alpha +2}.
\end{equation} 

\subsection{Fa\`a di Bruno's formula} \label{ssFdB} 
We recall that by the Fa\`a di Bruno formula (see Corollary~2.10 in~\cite{CoSa}), if $\beta $ is a multi-index, $ |\beta |\ge 1$, 
if $a\in C^{ |\beta |} (  O , \R) $ where $O$ is an open subset of $\R^N $, 
and if $\varphi \in C^{ |\beta |} ( U, \R )$ where $U$ is a neighborhood of $a(O)$, then on $O$, $\partial _x^\beta [ \varphi ( a (\cdot ) ) ]$ is a sum of terms of the form 
\begin{equation*} 
\varphi ^{(\nu)} ( a (\cdot ) ) \prod _{ \ell =1 }^{ |\beta |} (\partial _x^{\beta _\ell} a (\cdot ))^{\nu _\ell},
\end{equation*} 
with appropriate coefficients, where $\nu \in \{1, \cdots,  |\beta | \}$, $\nu _\ell \ge 0$, $\displaystyle \sum  _{ \ell=1 }^{  |\beta | }\nu_\ell = \nu $, $\displaystyle \sum _{ \ell=1 }^{  |\beta | } \nu _\ell \beta _\ell = \beta  $.

\section{The blow-up ansatz} \label{sBUAnsatz} 

In this section, we construct inductively an appropriate blow-up ansatz. 
The first ansatz is $U_0$ defined by~\eqref{fDUZ1} below. $U_0$ is a natural candidate, since it is an explicit blowing-up  solution of the ODE $\partial _t U_0= \lambda f(U_0)$. Moreover, the error term $i\Delta U_0$ is of lower order than both $\partial _t U_0$ and $ f(U_0) $. (See Lemma~\eqref{eEstUz} below.) However, we need at least the error term to be integrable in time near the singularity. Since $\Delta U_0$ is of order $(-t)^{ -\frac {2} {k} }  |U_0| \lesssim t^{- \frac {1} {\alpha } - \frac {2} {k}}$, this is not the case for any choice of $k$ if $\alpha \le 1$. 
In Section~\ref{ssrefined}, we introduce a procedure to reduce the singularity of the error term at any order of $(-t)$ by refining the approximate solution. 
This is important, not only to obtain blowup for arbitrarily small powers $\alpha $, but also to avoid any condition between $\Re \lambda $ and $\Im \lambda $.
We also point out that in this section, there is no condition on the power $\alpha $ other than $\alpha >0$.

Throughout this section, we assume  
\begin{equation} \label{feEstUz:1:3}
k> \max  \{ 6,  N\alpha \}. 
\end{equation} 
Let $A\in C^{k-1} (\R^N , \R)$ be piecewise of class $C^k$ and satisfy
\begin{equation} \label{on:A}
\begin{cases} 
A\geq 0 \hbox{ and }
|\partial_x^\beta A|\lesssim A^{1-\frac {|\beta|}{k}} &  \text{on }\R^N  \text{ for }  |\beta |\le k-1  ,\\
 A(x)=|x|^k &  \text{for }x\in \R^N ,  |x|\ge 2 . 
\end{cases} 
\end{equation} 
Assuming~\eqref{fCL1}, \eqref{feEstUz:1:3} and~\eqref{on:A}, and set 
\begin{equation} \label{fDUZ1}
U_0 (t, x) = (- \alpha t + A (x))^{- \frac {1 } {\alpha } - i \frac {\IL} {\alpha }}\quad t<0, \, x\in \R^N.
\end{equation} 
It follows that 
\begin{gather} 
 \text{$U_0$ is $C^\infty $ in $t<0$ and $C^{k-1}$ in $x\in \R^N $}, \label{fEUZ1:b1} \\
 \partial _t U_0 = \lambda f( U_0) , \label{fEUZ1} \\
  | U_0 | = (- \alpha t + A (x))^{- \frac {1 } {\alpha }} \le   (- \alpha t  )^{- \frac {1 } {\alpha }} . \label{fEUZ2}
\end{gather} 
Moreover,
\begin{equation} \label{fEUZ2:2}
\partial _t  | U_0 | = f(  |U_0| ) >  0 
\end{equation} 
and
\begin{equation} \label{fEUZ3}
A(x) \le  | U_0 |^{-\alpha } .
\end{equation} 

\subsection{Estimates of $U_0$}

\begin{lem} \label{eEstUz} 
Assume~\eqref{fCL1}, \eqref{feEstUz:1:3}, \eqref{on:A} and let
$U_0$ be given by~\eqref{fDUZ1}. If $p\ge 1$ then
\begin{equation}  \label{feEstUz:2} 
 \| U_0 (t)\| _{ L^p } \lesssim  (-t)^{-\frac {1} {\alpha }},
\end{equation} 
for  $-1 \le t<0$.
In addition, for every $\rho \in \R$, $ \ell \in \N$ and $ | \beta |\le k-1$,
\begin{align}  
 | \partial _t^\ell \partial _x^\beta U_0 |  & \lesssim  |U_0 |^{1 + \ell \alpha + \frac {\alpha } {k} |\beta |} \lesssim (-t)^{ - \ell  -\frac { |\beta |} {k} }  |U_0| , \label{feEstUz:1:0} \\
 | \partial _x^\beta ( | U_0|^\rho  )|  & \lesssim  |U_0 |^{\rho + \frac {\alpha } {k} |\beta |} \lesssim (-t)^{ -\frac { |\beta |} {k} }  |U_0|^\rho , \label{feEstUz:1:1} \\
 | \partial _x^\beta ( | U_0|^{\rho -1} U_0 )| & \lesssim  |U_0 |^{\rho + \frac {\alpha } {k} |\beta |} \lesssim (-t)^{ -\frac { |\beta |} {k} }  |U_0|^\rho , \label{feEstUz:1:2} 
\end{align} 
for $x\in \R^N $,  $t<0$, and
\begin{equation} \label{fUzHk1} 
U_0 \in C^\infty  ( (-\infty ,0), H^{k-1} (\R^N ) ).
\end{equation} 
Furthermore, for any $x_0\in \R^N$ such that $A(x_0)=0$, for any $r>0$, $-1 \le t<0$ and $1\le p\le \infty $,
\begin{equation}\label{e:24}
(- t)^{-\frac 1{\alpha }+\frac N{pk}}\lesssim \|U_0(t)\|_{L^p (|x-x_0|<r)} ,
\end{equation}
where the implicit constant in \eqref{e:24} depend on $r$ and $p$.
\end{lem} 

\begin{proof} 
Estimate~\eqref{feEstUz:2} is a consequence of~\eqref{fEUZ2}, \eqref{on:A} and~\eqref{feEstUz:1:3}. Indeed,
\begin{equation*} 
\int  _{ \R^N  }  |U_0|^p \le \int  _{  |x| <2 } |U_0|^p + \int  _{  |x| >2 } |U_0|^p \lesssim (-t)^{-  \frac {p } {\alpha }}+  \int  _{  |x| >2 }  |x|^{- \frac {pk} {\alpha }}  \lesssim (-t)^{-  \frac {p} {\alpha }}.
\end{equation*} 

We now set 
\begin{equation*} 
U_0= W^{-\frac {1} {\alpha }- i \frac {\IL} {\alpha }}  \text{ where }  W= - \alpha t + A (x) >0,
\end{equation*} 
and we prove~\eqref{feEstUz:1:1}-\eqref{feEstUz:1:2}. We let $ |\beta |\ge 1$, and we write $ |U_0|^\rho = \varphi ( W )$ with $\varphi (s)= s^{- \frac {\rho } {\alpha }}$.
We see that $ |\varphi ^{(\nu)} (s) | \lesssim  s^{- \frac {\rho } {\alpha }- \nu}$ for $s>0$. Moreover, if $1\le  |\gamma |\le  |\beta |$, then
\begin{equation}  \label{feEstUz:3} 
 |\partial _x^{\gamma } W | \lesssim  |\partial _x^{\gamma } A| \lesssim A^{1-\frac {| \gamma |}{k}}
 \lesssim W^{1-\frac {| \gamma |}{k}},
\end{equation} 
where we used~\eqref{on:A}, $A\le W$, and $ 1 -\frac {| \gamma   |}{k} \ge  1 -\frac {| \beta  |}{k} \ge \frac {1} {k}>0$. 
By the Fa\`a di Bruno formula (see Section~\ref{ssFdB}), we deduce that $  | \partial _x^\beta ( | U_0|^\rho  )|  $ is estimated by a sum of terms of the form 
\begin{equation*} 
B= W ^{- \frac {\rho } {\alpha }- \nu}  \prod _{ \ell =1 }^{ |\beta |} W^{(1-\frac {| \beta _\ell |}{k}) \nu _\ell} ,
\end{equation*} 
with appropriate coefficients, where $\nu \in \{1, \cdots,  |\beta | \}$, $\nu _\ell \ge 0$, $\displaystyle \sum  _{ \ell=1 }^{  |\beta | }\nu_\ell = \nu $, $\displaystyle \sum _{ \ell=1 }^{  |\beta | } \nu _\ell \beta _\ell = \beta  $.
It follows that
\begin{equation*} 
B = W ^{- \frac {\rho } {\alpha }- \nu}   W^{\nu -\frac {| \beta  |}{k} } =  W ^{- \frac {\rho } {\alpha } -\frac {| \beta  |}{k}} =   |U_0 |^{\rho + \frac {\alpha } {k} |\beta |} ,
\end{equation*} 
and~\eqref{feEstUz:1:1} follows.

Now we claim that if $1 \le | \widetilde{\beta } | \le k-1$, then
\begin{equation}  \label{feEstUz:2b} 
  | \partial _x^{ \widetilde{\beta } }U_0 |  \lesssim  |U_0 |^{ 1+ \frac {\alpha } {k} |\widetilde{\beta }  |} .
\end{equation} 
Indeed, we write
\begin{equation}  \label{feEstUz:2:2} 
U_0 =  |U_0|  \Bigl(  \cos  \Bigl( \frac {\IL } {\alpha } \log W \Bigr) + i \sin  \Bigl( \frac {\IL } {\alpha } \log W \Bigr) \Bigr)
\end{equation} 
Since $(\log s)^{(\nu )} \lesssim s^{- \nu}$ for $\nu\ge 1$, it follows easily from Fa\`a di Bruno's formula and~\eqref{feEstUz:3} that 
\begin{equation}  \label{feEstUz:4} 
 |\partial _x^\gamma (\log W)| \lesssim W^{- \frac { |\gamma |} {k}}
\end{equation} 
if $1\le  |\gamma |\le  |\widetilde{\beta }  |$. 
Using again Fa\`a di Bruno's formula together with~\eqref{feEstUz:4}, we obtain
\begin{equation}  \label{feEstUz:5} 
 \Bigl| \partial _x^\gamma  \Bigl(  \cos  \Bigl( \frac {\IL } {\alpha } \log W \Bigr) \Bigr) \Bigr| 
 +  \Bigl| \partial _x^\gamma  \Bigl(  \sin  \Bigl( \frac {\IL } {\alpha } \log W \Bigr) \Bigr) \Bigr| \lesssim W^{- \frac { |\gamma |} {k}}
\end{equation} 
if $1\le  |\gamma |\le  |\widetilde{\beta }  |$. 
Estimate~\eqref{feEstUz:2b} follows from~\eqref{feEstUz:2:2}, \eqref{feEstUz:1:1}  with $\rho =1$, \eqref{feEstUz:5}, and Leibnitz's formula.   
 
Estimate~\eqref{feEstUz:1:2}  follows from~\eqref{feEstUz:1:1}, \eqref{feEstUz:2b}, and Leibnitz's formula.

To prove~\eqref{feEstUz:1:0}, we observe that $\partial _t^\ell U_0 = c W^{ - \frac {1} {\alpha }- \ell - i \frac {\IL} {\alpha }}$ for some constant $c$. Therefore,  $\partial _t^\ell U_0 = c  \widetilde{U}_0   $, where $ \widetilde{U}_0= W^{ - {1} /{ \widetilde{\alpha } } - i  { \widetilde{\IL} } / { \widetilde{\alpha } } } $ with $ \widetilde{\alpha } = \frac {\alpha } {1+ \ell \alpha }$ and $ \widetilde{\IL } = \frac {\IL } {1+ \ell \alpha }$. 
In particular, $  | \widetilde{U}_0 | =  |U_0|^{1+ \ell \alpha }$. Applying formula~\eqref{feEstUz:2b} with $\alpha $ and $\IL$ replaced by $ \widetilde{\alpha } $ and  $ \widetilde{\IL} $, we obtain
\begin{equation*} 
  | \partial _x^{ {\beta } } \widetilde{U} _0 |  \lesssim  | \widetilde{U} _0 |^{ 1+ \frac { \widetilde{\alpha}  } {k} |{\beta }  |} =  |U_0|^{ 1+ \ell \alpha + \frac {\alpha } {k}  |\beta |},
\end{equation*} 
from which~\eqref{feEstUz:1:0} follows.

Property~\eqref{fUzHk1} is an immediate consequence of~\eqref{fEUZ1:b1}, \eqref{feEstUz:1:0} and~\eqref{feEstUz:2}. 

Finally, we prove~\eqref{e:24}. Let $x_0\in \R^N$ be such that $A(x_0)=0$ and $r>0$.
We have $ |U_0 ( t, x_0)| = (-\alpha t)^{-\frac {1} {\alpha }}$, and~\eqref{e:24} follows in the case $p= \infty $. We now assume $p<\infty $. 
Since $A$ is piecewise $C^k$, it follows easily from~\eqref{on:A} that for any $x$ such that $|x-x_0|<r$, we have $|A(x)|\leq C(r) |x-x_0|^k$; and so
\begin{equation*} 
 |U_0 (t, x)|^p \gtrsim (-  t +  |x-x_0|^k )^{- \frac {p } {\alpha }}
\end{equation*} 
Estimate~\eqref{e:24} then follows from
\begin{equation} \label{label4}
\begin{split} 
\int_{|y |<r} ( - t +| y |^k)^{-\frac{ p }{\alpha }} dy&=
 (- t)^{-\frac {p} {\alpha }+\frac Nk} \int_{|z|< r (- t) ^{-\frac1k}} (1+|z|^k)^{-\frac{p}{\alpha }} dz\\
&\gtrsim
(-t )^{-\frac {p} {\alpha }+\frac Nk}\int_{|z|< r} (1+|z|^k)^{-\frac{p}{\alpha }} dz
\gtrsim (- t) ^{-\frac {p} {\alpha }+\frac Nk}.
\end{split} 
\end{equation} 

This completes the proof.
\end{proof} 

\subsection{The refined blow-up ansatz} \label{ssrefined} 

 We consider the linearization of equation~\eqref{fEUZ1} 
 \begin{equation} \label{fRA1} 
 \partial _t w = \lambda \frac {\alpha +2} {2}  |U_0|^\alpha w +  \lambda \frac {\alpha } {2}  |U_0|^{\alpha -2} U_0^2  \overline{w} 
 = \lambda df( U_0) w,
 \end{equation}  
 where $df$ is defined by~\eqref{fDerf2}. 
Equation~\eqref{fRA1}  has the two solutions $w= iU_0$ and $w= \partial _t U_0$, i.e. $w= \lambda  |U_0|^\alpha U_0$.
Moreover, $\partial _t (  |U_0|^\alpha U _0) = (\lambda +\alpha )  |U_0|^{2\alpha } U_0$.
Elementary calculations show that for suitable $G$, the corresponding nonhomogeneous equation 
 \begin{equation} \label{fRA2} 
 \partial _t w =  \lambda df(U_0) w + G
 \end{equation}
has the solution $w= \Foru (G) $, where
\begin{equation}  \label{fRA3} 
\begin{split} 
\Foru ( G)  = & |U_0|^\alpha U_0 \int _0^t [  |U_0|^{-\alpha -2} \Re ( \overline{U_0} G ) ] (s) \, ds
+ i  U_0 \int _0^t [  |U_0|^{-2} \Im ( \overline{U_0} G ) ] (s) \, ds \\
& + i \alpha \IL  U_0 \int _0^t   |U_0(s)|^{2\alpha } \int _0^s [  |U_0|^{-\alpha -2} \Re ( \overline{U_0} G ) ] (\sigma ) \, d\sigma ds .
\end{split} 
\end{equation} 
We define $U_j, w_j, \Err _j$ by
\begin{equation}  \label{fRA5} 
w_0= iU_0,\quad \Err _0= -\partial _t U_0+ i \Delta U_0 + \lambda f( U_0) = i\Delta U_0,
\end{equation} 
and then recursively
\begin{equation}  \label{fRA6} 
w_j= \Foru (\Err  _{ j-1 }),\quad U_j= U _{ j-1 }+ w_j \quad \Err _j= -\partial _t U_j+ i \Delta U_j + \lambda f ( U_j ),
\end{equation}
for $j \ge 1$, as long as this makes sense. 
We will see that for $j\le \frac {k-4} {2}$, $\Foru (\Err  _{ j-1 })$ is well defined at each step,  on a sufficiently small time interval.
We have the following estimates.

\begin{lem} \label{eEstUJ} 
Assume~\eqref{fCL1}, \eqref{feEstUz:1:3}, \eqref{on:A}, and let $U_j, w_j, \Err _j$ be given by~\eqref{fDUZ1}, \eqref{fRA5} and~\eqref{fRA6}.  
There exists $-1\le T<0$ such that the following estimates hold for all $0\le j\le \frac {k-4} {2}$.
\begin{enumerate}[{\rm (i)}] 

\item \label{eEstUJ:1} 
If $0\le  |\beta |\le k-1- 2j$, then
\begin{equation} \label{feEstUJ:1}
 |\partial _x^\beta w_j| \lesssim (-t) ^{j(1- \frac {2} {k}) - \frac { |\beta |} {k}}  |U_0 | , \quad T\le t<0, x\in \R^N .
\end{equation} 

\item \label{eEstUJ:2} 
 If $0\le  |\beta |\le k-3- 2j$, then
\begin{equation} \label{feEstUJ:2}
 |\partial _x^\beta \Err _j| \lesssim (-t) ^{j(1- \frac {2} {k}) - \frac { |\beta |+2 } {k}}  |U_0 | , \quad T\le t<0, x\in \R^N .
\end{equation} 

\item \label{eEstUJ:3} 
If $0\le  |\beta |\le k-1-2j$, then
\begin{equation} \label{feEstUJ:3}
  |\partial _x^\beta (U_j - U_0) | \lesssim (-t)^{1- \frac { |\beta |+ 2} {k}}  |U_0| , \quad T\le t<0, x\in \R^N .
\end{equation} 
\end{enumerate} 

Moreover
\begin{equation} \label{feEstUJ:4}
\frac {1} {2}  |U_0| \le  | U _j| \le 2  |U_0|, \quad T\le t<0, x\in \R^N .
\end{equation} 
In addition, 
\begin{equation} \label{feEstUJ:5}
U_j \in C^1 ((T, 0), H^{ k-1-2j } (\R^N ) ) .
\end{equation} 

\end{lem} 

\begin{proof}
For $j=0$, estimates~\eqref{feEstUJ:1} and~\eqref{feEstUJ:2} are immediate consequences of~\eqref{feEstUz:1:2} and~\eqref{fEUZ2} (for $T=-\infty $), and estimates~\eqref{feEstUJ:3} and~\eqref{feEstUJ:4} are trivial. 
We now proceed by induction on $j$. 
We assume that for some $ 1\le  j \le \frac {k-4} {2}$, estimates~\eqref{feEstUJ:1}--\eqref{feEstUJ:4} hold for $0, \cdots, j-1$,
 and we prove estimates~\eqref{feEstUJ:1}--\eqref{feEstUJ:4} for $j$, by possibly assuming $T$ smaller.

 {\sl Proof of~\eqref{feEstUJ:1}}. 
 Let $0\le  |\beta |\le k-1-2j$.
 Given $\rho \in \R$, it follows from Leibnitz's formula, \eqref{feEstUz:1:2}, and~\eqref{feEstUJ:2} for $j-1$ that
\begin{equation*} 
\begin{split} 
 |\partial _x^\beta ( |U_0|^\rho  \overline{U_0}  \Err  _{ j-1 })| &  \lesssim \sum_{ \beta _1+ \beta _2 =\beta  }  |\partial _x^{\beta_1} ( |U_0|^\rho  U_0)| \, |\partial _x^{\beta_2} \Err  _{ j-1 }| \\ &  \lesssim \sum_{ \beta _1+ \beta _2 =\beta  } (-t)^{ -\frac { |\beta _1|} {k} }  |U_0|^{\rho+1} (-t) ^{(j-1)(1- \frac {2} {k}) - \frac { |\beta _2|+2 } {k}}  |U_0 |  \\ &  \lesssim   (-t)^{  j(1- \frac {2} {k}) - \frac { |\beta | } {k}-1 }  |U_0|^{\rho+2} .
\end{split} 
\end{equation*} 
Integrating on $(t, 0)$ for $t\in (T,0)$, and using that $ |U_0|^{-1}$ is a decreasing function of $t$ by~\eqref{fEUZ2:2}, we see that if $\rho +2\le 0$, then
\begin{equation*} 
 \Bigl| \partial _x^\beta \int _0^t |U_0|^\rho  \overline{U_0}  \Err  _{ j-1 } \,ds \Bigr|  \lesssim  | U_0 (t)|^{\rho +2} \int _t^0 (-s)^{  j(1- \frac {2} {k}) - \frac { |\beta | } {k}-1 } ds.
\end{equation*} 
Note that $j\ge 1$ and $\frac { |\beta |} {k} \le 1- \frac {2j+1} {k}$, so that $j(1- \frac {2} {k}) - \frac { |\beta | } {k} \ge 1- \frac {2} {k} - \frac { |\beta | } {k} >0$; and so,
\begin{equation} \label{fRA7} 
 \Bigl| \partial _x^\beta \int _0^t |U_0|^\rho  \overline{U_0}  \Err  _{ j-1 } \,ds \Bigr|  \lesssim  (-t)^{  j(1- \frac {2} {k}) - \frac { |\beta | } {k} }  | U_0 (t)|^{\rho +2}  .
\end{equation} 
It follows from Leibnitz's formula, \eqref{feEstUz:1:2}, \eqref{fRA7}, \eqref{feEstUz:1:1} and~\eqref{fEUZ2} that
\begin{equation*} 
\begin{split} 
 |\partial _x^\beta w_j|  \lesssim & \sum_{ \beta _1+\beta _2= \beta  } (-t)^{ -\frac { |\beta _1|} {k} }  |U_0|^{\alpha +1}   (-t)^{  j(1- \frac {2} {k}) - \frac { |\beta _2| } {k} }  | U_0 |^{- \alpha } \\ & + \sum_{ \beta _1+\beta _2= \beta  } (-t)^{ -\frac { |\beta _1|} {k} }  |U_0|   (-t)^{  j(1- \frac {2} {k}) - \frac { |\beta _2| } {k} }  \\ & + \sum_{ \beta _1+\beta _2+ \beta _3= \beta   }  (-t)^{ -\frac { |\beta _1|} {k} }  |U_0| \int _t^0 (-s)^{ _{ \frac { |\beta _2|} {k}-1 }}  |U_0 |^\alpha (-s)^{  j(1- \frac {2} {k}) - \frac { |\beta_3 | } {k} }  | U_0 |^{-\alpha } ds  \\  \lesssim & (-t)^{  j(1- \frac {2} {k}) - \frac { |\beta | } {k} }  | U_0 | \\ &+ \sum_{ \beta _1+\beta _2+ \beta _3= \beta   }  (-t)^{ -\frac { |\beta _1|} {k} }  |U_0| \int _t^0  (-s)^{  j(1- \frac {2} {k}) - \frac { |\beta _2+ \beta_3 | } {k} -1} ds\\ 
 \lesssim & (-t)^{  j(1- \frac {2} {k}) - \frac { |\beta | } {k} }  | U_0 |. 
\end{split} 
\end{equation*} 
This proves~\eqref{feEstUJ:1}.

 {\sl Proof of~\eqref{feEstUJ:3} and~\eqref{feEstUJ:4}}. 
Since $U_j- U_0 = w_j + U _{ j-1 }- U_0$, estimate~\eqref{feEstUJ:3} for $j$ follows from~\eqref{feEstUJ:1} for $j$ and~\eqref{feEstUJ:3} for $j-1$. 
Estimate~\eqref{feEstUJ:4} follows from~\eqref{feEstUJ:3} by possibly choosing $T>0$ smaller.
 
 {\sl Proof of~\eqref{feEstUJ:2}}. 
Since $U_j- U _{ j-1 } = w_j$,  it follows from~\eqref{fRA6} and the definition of $\Foru$ that
 \begin{equation*} 
 \begin{split} 
 \Err _j - \Err  _{ j-1 }& = -\partial _t w_j + i\Delta w_j + \lambda  f(U_j) - \lambda  f( U _{ j-1 }) \\ &
 = -\lambda df(U_0) w_j- \Err _{ j-1 } + i\Delta w_j + \lambda  f(U_j) - \lambda  f( U _{ j-1 }),
 \end{split} 
 \end{equation*}  
so that
\begin{equation*} 
\begin{split} 
 \Err _j = &  i\Delta w_j + \lambda  [ f(U _{ j-1 }+ w_j) - f( U _{ j-1 })- df(U _0 ) w_j ] \\
 = &:  A_1 +  \lambda  A_2  .
\end{split} 
\end{equation*} 
It follows from~\eqref{feEstUJ:1}  (for $j$) that if $ |\beta |\le k-3- 2j$ (so that $ |\beta | +2\le k-1- 2j$), then 
\begin{equation} \label{fRA8} 
 |\partial _x^\beta A_1 | \lesssim (-t) ^{j(1- \frac {2} {k}) - \frac { |\beta | +2} {k}}  |U_0 | .
\end{equation} 
We now estimate $A_2$, and we write
\begin{equation*} 
f(U _{ j-1 }+ w_j) - f( U _{ j-1 })  = \int _0^1 \frac {d} {d\theta } f( U_0+ g_j (\theta ) ) \, d\theta 
  = \int _0^1 df( U_0+g_j (\theta ) )w_j  \, d\theta ,
\end{equation*} 
where
\begin{equation*} 
g_j (\theta ) = U _{ j-1 }- U_0 + \theta w_j ,
\end{equation*} 
so that 
\begin{equation*} 
\begin{split} 
A_2  = & \int _0^1 [ df( U_0+ g_j (\theta ) )w_j - df (U_0) w_j ]
\, d\theta \\
= & \frac {\alpha +2} {2} \int _0^1 ( |U_0+ g_j (\theta ) |^\alpha - |U_0 |^\alpha ) w_j\, d\theta 
\\ & + \frac {\alpha } {2} \int _0^1 ( |U_0+ g_j (\theta ) |^{\alpha -2} (U_0+ g_j (\theta ))^2 - |U_0 |^{\alpha -2} U_0^2 )  \overline{w_j} \, d\theta .
\end{split} 
\end{equation*} 
We write
\begin{equation*} 
\begin{split} 
|U_0+ g_j (\theta ) |^\alpha - |U_0 |^\alpha & = \int _0^1 \frac {d } {d\tau }   ( \tau  |U_0+ g_j(\theta )|^2 + (1-\tau )  |U_0|^2 )^{\frac {\alpha } {2}} d\tau  \\ & = \frac {\alpha } {2} \eta [ |U_0+ g_j(\theta )|^2 - |U_0|^2 ] 
 \\ & = \frac {\alpha } {2} \eta [ 2\Re ( \overline{U_0} g_j(\theta )) + |g_j(\theta )|^2 ]
\end{split} 
\end{equation*} 
where
\begin{equation*} 
\eta = \int _0^1  ( \tau  |U_0+ g_j(\theta )|^2 + (1-\tau )  |U_0|^2 )^{\frac {\alpha } {2} -1 } d\tau .
\end{equation*} 
Similarly,
\begin{equation*} 
\begin{split} 
 |U_0+ & g_j (\theta ) |^{\alpha -2} (U_0  + g_j (\theta ))^2 - |U_0 |^{\alpha -2} U_0^2   \\
&= (  |U_0+ g_j (\theta ) |^{\alpha -2} -  |U_0 |^{\alpha -2}) (U_0+ g_j (\theta ))^2 + |U_0 |^{\alpha -2} (  (U_0+ g_j (\theta ))^2 -U_0^2 ) \\ & 
= \frac {\alpha -2} {2}  \widetilde{\eta }   [ 2\Re ( \overline{U_0} g_j(\theta )) + |g_j(\theta )|^2 ] + 
|U_0 |^{\alpha -2} (2U_0+ g_j (\theta ))g_j (\theta ) ,
\end{split} 
\end{equation*} 
where
\begin{equation*} 
 \widetilde{\eta}  = (U_0  + g_j (\theta ))^2 \int _0^1  ( \tau  |U_0+ g_j(\theta )|^2 + (1-\tau )  |U_0|^2 )^{\frac {\alpha  } {2} -2} d\tau  .
\end{equation*} 
Thus we may write
\begin{equation} \label{fRA9} 
\begin{split} 
A_2 = & \frac {\alpha (\alpha +2)} {4} \int _0^1 \eta [ 2\Re ( \overline{U_0} g_j(\theta )) + |g_j(\theta )|^2 ]w_j \, d\theta \\ &
+ \frac {\alpha (\alpha -2)} {4} \int _0^1 \widetilde{\eta }   [ 2\Re ( \overline{U_0} g_j(\theta )) + |g_j(\theta )|^2 ]  \overline{w_j} \, d\theta  \\ &+ \frac {\alpha } {2} \int _0^1 |U_0 |^{\alpha -2} (2U_0+ g_j (\theta ))g_j (\theta )  \overline{w_j} \, d\theta \\ = & \frac {\alpha (\alpha +2)} {4} B_1 +  \frac {\alpha (\alpha -2)} {4}B_2 + \frac {\alpha } {2} B_3 . 
\end{split} 
\end{equation} 
Using~\eqref{feEstUJ:1}, we obtain by choosing $T$ possibly smaller 
\begin{equation*} 
| g_j (\theta ) | \le \sum_{ \ell=1 }^j  |w_\ell | \le C (-t) ^{1- \frac {2} {k} }  |U_0 | \le \frac {1} {2}  |U_0|,
\end{equation*} 
so that 
\begin{equation} \label{fRA10} 
\frac {1} {4} |U_0|^2 \le  \tau  |U_0+ g_j(\theta )|^2 + (1-\tau )  |U_0|^2 \le 3  |U_0|^2
\end{equation} 
for all $0\le \tau, \theta  \le 1$. 
Applying~\eqref{feEstUz:1:1}, \eqref{feEstUz:1:2}, \eqref{feEstUJ:1}, and Leibnitz's formula, it is not difficult to show that 
if $ |\beta |\le k-3- 2j$ then
\begin{equation} \label{fRA11} 
 |\partial _x^\beta (  |U_0+ g_j(\theta )|^2 + (1-\tau )  |U_0|^2 ) | \lesssim  (-t)^{- \frac { |\beta |} {k}}  |U_0|^2 .
\end{equation} 
Using now~\eqref{fRA10}, \eqref{fRA11}, and the Fa\`a di Bruno formula, we deduce that  
\begin{equation} \label{fRA12} 
 | \partial _x ^\beta \eta | \lesssim (-t)^{- \frac {   |\beta  |} {k}}    |U_0|^{ \alpha - 2  } .
\end{equation} 
Similarly (using in addition Leibnitz's formula), we see that 
\begin{equation} \label{fRA13} 
 | \partial _x ^\beta  \widetilde{\eta}  | \lesssim (-t)^{- \frac {   |\beta  |} {k}}    |U_0|^{ \alpha - 2  } .
\end{equation} 
Next, we deduce from~\eqref{feEstUz:1:2}, \eqref{feEstUJ:1}, \eqref{fRA12}, \eqref{fRA13}, \eqref{fEUZ2}, and Leibnitz's formula that 
if $ |\beta |\le k-3- 2j$ then
\begin{equation} \label{fRA14} 
 |\partial _x^\beta B_1 | +  |\partial _x^\beta B_2 | \lesssim  (-t) ^{j(1- \frac {2} {k}) - \frac { |\beta |+2} {k}}  |U_0 | .
\end{equation}   
Using~\eqref{feEstUz:1:1} with $\rho = \alpha -2$, we obtain similarly
\begin{equation} \label{fRA15} 
 |\partial _x^\beta B_3 | \lesssim 
  (-t) ^{j(1- \frac {2} {k}) - \frac { |\beta |+2} {k}}  |U_0 | .
\end{equation}
Estimate~\eqref{feEstUJ:2} follows from~\eqref{fRA8}, \eqref{fRA9},  \eqref{fRA14} and~\eqref{fRA15}. 

Finally, we prove~\eqref{feEstUJ:5}.
For this, we prove by induction on $j$ that 
\begin{equation} \label{fRT1} 
U_j, w_j\in C^1((T,0), H^{k -1-2j}  (\R^N ) )  \text{ and } \Err _j \in C ((T,0), H^{k -3-2j}  (\R^N ) ).
\end{equation} 
For $j=0$, \eqref{fRT1} holds, by~\eqref{fUzHk1}. 
We assume that for some $ 1\le  j \le \frac {k-4} {2}$, property~\eqref{fRT1} holds for $ j-1$,
and we prove it for $j$. 
Let $t_0\in (T, 0)$. It follows from~\eqref{feEstUJ:1} and~\eqref{feEstUz:2} that $w_j (t_0) \in H^{k -1-2j} (\R^N ) $. 
Moreover, $\Err  _{ j-1 } \in C ((T,0), H^{k -1-2j}  (\R^N ) )$ by the induction assumption. 
Since $\partial _t w_j=    \lambda df(U_0) w_j + \Err  _{ j-1 }$, it is not difficult to prove (using Lemma~\ref{eEstUz} for the relevant estimates of $U_0$)  that $w_j \in  C^1((T,0), H^{k -1-2j}  (\R^N ) )  $. 
Hence $U_j \in  C^1((T,0), H^{k -1-2j}  (\R^N ) )  $, and by definition of $\Err _j$, we deduce that $\Err _j \in C ((T,0), H^{k -3-2j}  (\R^N ) )$. This proves~\eqref{fRT1}.
\end{proof} 

\section{Construction and estimates of approximate solutions} \label{sApprox} 

In this section, we construct a sequence $u_n$ of solutions of~\eqref{NLS1}, close to the ansatz $U_J$ of Section~\ref{sBUAnsatz}, which will eventually converge to the blowing-up solution  of Theorem~\ref{eThm1}. We estimate  $\varepsilon _n= u_n - U_J$ by an energy method. 
More precisely, we estimate $ (-t)^{- \sigma } \| \varepsilon _n \| _{ L^2 } + (-t)^{ - (1 -\theta )\sigma  } \| \nabla \varepsilon _n \| _{ L^2 } $ for some appropriate parameters $\sigma \ge 0$ and $0\le \theta \le 1$.
This parameter $\sigma $ is taken large enough to avoid unnecessary condition on $\lambda $, see~\eqref{fCite1} and~\eqref{rfAPE17}. 
Moreover, the parameter $J$ of the ansatz $U_J$ is chosen sufficiently large to absorb the singularity $(-t)^{- \sigma }$, see~\eqref{fCite2} and~\eqref{rfAPE7}.

We now go into details. We define $\sigma , \theta >0$ by
\begin{align} 
\sigma &= \max \{ 2^{\alpha +1}\alpha ^{-1} |\lambda | \Csu , 2 \alpha  |\lambda | (8 \Csu  |\lambda |)^{\alpha } \} , \label{feApprox1:3} \\
 \theta &= \min \Bigl\{ \frac {1} {2}, \frac {2} {N}, \frac {1} {N\alpha +1}, \frac {1} {3 \alpha ^2} \Bigr\}  , \label{feApprox1:5}
\end{align} 
where $\Csu$ is given by Lemma~$\ref{eANL1}$, and we set 
\begin{gather} 
J = \Bigl[  \frac {2} {\alpha } + 4 \sigma \Bigr] + 1 , \label{feApprox1:6} \\
k = \max  \Bigl\{   2 J + 4 , \frac {4} {\theta \sigma }, 2 N\alpha \Bigr\} .  \label{feApprox1:7}
 \end{gather} 
 In particular, $k$ satisfies~\eqref{feEstUz:1:3}. 
We let $A\in C^{k-1} (\R^N , \R)$ be piecewise of class $C^k$ and satisfy~\eqref{on:A},
and we consider the ansatz $U_J$ constructed in Section~\ref{sBUAnsatz}, and $T <0$ given by Lemma~\ref{eEstUJ}.
(This is possible since $2J \le k-4$ by~\eqref{feApprox1:7}.)
For $n > - \frac {1} {T}$, we set
\begin{equation*} 
T_n= - \frac {1} {n}\in  (T, 0). 
\end{equation*} 
Since $U_J  (T_n) \in H^2 (\R^N ) $ (by~\eqref{feEstUJ:5} and~\eqref{feApprox1:7}) it follows that there exist $s _n < T_n$ and a unique solution $u_n\in C( (s _n, T_n] , H^2 (\R^N ) ) \cap C^1( (s _n, T_n] , L^2 (\R^N ) )$ of
\begin{equation} \label{fAP1} 
\begin{cases} 
\partial _t u_n = i\Delta u_n + \lambda f  (u_n) \\  u_n (T_n) = U_J(T_n) ,
\end{cases} 
\end{equation} 
defined on the maximal interval $( s _n, T_n]$, with the blow-up alternative that if $s _n >-\infty $, then
\begin{equation} \label{fBUA1} 
  \| \nabla u_n (t)  \| _{ L^2 } \goto _{ t \downarrow s _n} \infty .  
\end{equation} 
See~\cite{Kato1}.

We let $\varepsilon _n \in C( ( \max \{ s_n, T \}, T_n], H^2 (\R^N ) ) \cap C^1( ( \max \{ s_n, T \}, T_n], L^2 (\R^N ) )$ be defined by
\begin{equation} \label{fDfnen} 
u_n = U_J + \varepsilon _n 
\end{equation} 
and we have the following estimate.

\begin{prop} \label{eApprox1} 
Assume~\eqref{feApprox1:3},  \eqref{feApprox1:5}, \eqref{feApprox1:6} and~\eqref{feApprox1:7}.
If $\varepsilon _n$ is given by~\eqref{fDfnen}, then 
there exist  $T\le S < 0$ and $n_0 >  - \frac {1} { S}$ such that
\begin{equation} \label{feApprox1:10} 
s_n < S, \quad  \text{for all } n\ge n_0.
\end{equation} 
Moreover,
\begin{gather} 
 \| \varepsilon _n (t) \| _{ L^2  } \le (-t)^\sigma \label{feApprox1:1} \\ 
  \| \nabla \varepsilon _n (t) \| _{ L^2  } \le (-t)^{(1-\theta )\sigma } \label{feApprox1:2} 
\end{gather} 
for all $n\ge n_0$ and $S \le t\le T_n$, and
\begin{equation} \label{feApprox1:9} 
\int  _S ^{T_n} \int  _{ \R^N  }  |\varepsilon _n |^\alpha  |\nabla \varepsilon _n |^2 \le 1. 
\end{equation} 
\end{prop} 

\begin{proof} 
Throughout the proof, we write $\varepsilon $ instead of $\varepsilon _n$. 
Moreover, $C$ denotes a constant that may change from line to line, but that is independent of $n$ and $t$.
Unless otherwise specified, all integrals are over $\R^N $.
Using~\eqref{fAP1} and~\eqref{fRA6}, we have
\begin{equation} \label{fAP2} 
\begin{cases} 
\partial _t \varepsilon = i\Delta \varepsilon  + \lambda ( f (U_J+ \varepsilon ) - f(U_J) ) + \Err_J \\
\varepsilon (T_n)= 0 .
\end{cases} 
\end{equation}  
We control $\varepsilon $ by energy estimates.
Let
\begin{equation} \label{fAP3} 
\begin{split} 
\tau _n & = \inf \{ t\in [ \max\{ T, s_n \},T_n];\,  \| \varepsilon  (s) \| _{ L^2  } \le (-s)^\sigma  \text{ and }  \\
&  \| \nabla \varepsilon  (s) \| _{ L^2  } \le (-s)^{(1-\theta )\sigma }  \text{ for }t < s\le T_n ,
 \text{ and }  \int  _t ^{T_n} \int  _{ \R^N  }  |\varepsilon |^\alpha  |\nabla \varepsilon |^2 \le 1 \} .
\end{split} 
\end{equation} 
Since $\varepsilon (T_n)=0$, we see that $T\le \tau _n < T_n$. 
Moreover, it follows from the blowup alternative~\eqref{fBUA1} that
\begin{equation}  \label{fAP3:b1} 
s_n < \tau _n .
\end{equation} 
In addition, by Gagliardo-Nirenberg's inequality, \eqref{fAP3}  and~\eqref{feApprox1:5}, 
\begin{equation} \label{fAP4} 
 \| \varepsilon (t) \| _{ L^{\alpha +2} }^{ \alpha +2} \le C (-t)^{ ( \alpha +2 - \frac {N\alpha } {2} \theta )\sigma }
 = C (-t)^{ 2\sigma + \alpha ( 1 - \frac {N\theta  } {2} )\sigma } \le C (-t)^{ 2\sigma },  
\end{equation} 
for $\tau _n \le t\le T_n$. 
Moreover, it follows from~\eqref{feEstUJ:4}, \eqref{fEUZ2}, \eqref{feEstUJ:3} and~\eqref{feEstUz:1:0}  that
\begin{equation} \label{fEUGU1} 
 | U_J| \le 2  |U_0| \le 2 \alpha ^{-\frac {1} {\alpha }} (-t)^{-\frac {1} {\alpha }}, \quad  |\nabla U_J| \le C (-t)^{- \frac {1} {k}}  |U_0|
 \le C (-t)^{-\frac {1} {\alpha } - \frac {1} {k}}
\end{equation} 
for all $T<t<0$.

We first estimate $ \|\varepsilon \| _{ L^2 }$. Multiplying~\eqref{fAP2} by $ \overline{\varepsilon } $ and taking the real part, we obtain
\begin{equation*} 
\begin{split} 
\frac {1} {2} \frac {d} {dt}  \|\varepsilon (t)\| _{ L^2 }^2  = &  \Re  \Bigl(  \lambda   \int    [ df(U_J ) \varepsilon  ]\overline{\varepsilon }  \Bigr) \\ &+  \Re   \Bigl( \lambda \int   [ f (U_J+ \varepsilon ) - f(U_J) - df(U_J ) \varepsilon  ]  \overline{\varepsilon }  \Bigr) + \Re \int   \Err _J  \overline{\varepsilon}   .
\end{split} 
\end{equation*} 
Using~\eqref{feANL1:1} and~\eqref{fDerf5}, we deduce that
\begin{equation*} 
\begin{split} 
 \Bigl| \frac {d} {dt}  \|\varepsilon (t)\| _{ L^2 }^2 \Bigr| \le & 2  |\lambda | \Csu \int  |U_J |^\alpha  |\varepsilon |^2 + 2  |\lambda | \Csu \int  [ |\varepsilon |^{\alpha +2} + \Calpha  |U_J|^{\alpha -1}  |\varepsilon |^3 ] \\ & +  \| \Err _J\| _{ L^2 }  \|\varepsilon \| _{ L^2 } .
\end{split} 
\end{equation*} 
By~\eqref{fEUGU1},
\begin{equation*} 
2  |\lambda | \Csu \int  |U_J |^\alpha  |\varepsilon |^2 \le 2^{\alpha +1}\alpha ^{-1} |\lambda | \Csu  (-t)^{-1}  \| \varepsilon \| _{ L^2 }^2 .
\end{equation*}  
The term $\int  |\varepsilon |^{ \alpha +2} $ is estimated by~\eqref{fAP4}. 
Note that $\Calpha  \not = 0$ only if $\alpha >1$. In this case $\alpha -1>0$ and $2 <3<\alpha +2$, so we deduce from~\eqref{fEUGU1}, H\"older's inequality, \eqref{fAP3}, \eqref{fAP4} and~\eqref{feApprox1:5} that
\begin{equation*} 
\Calpha  \int  |U_J |^{\alpha -1}  |\varepsilon |^3 \le C \Calpha  (-t) ^{-1 + \frac {1} {\alpha }} \int  |\varepsilon |^3
 \le C  (-t) ^{-1 + \frac {1} {\alpha } + 2 \sigma  } .
\end{equation*} 
Next, by~\eqref{feEstUJ:2}, \eqref{feEstUz:2}  and~\eqref{fAP3},
\begin{equation*} 
\| \Err _J\| _{ L^2 }  \|\varepsilon \| _{ L^2 } \le C (-t) ^{ J (1- \frac {2} {k}) - \frac {2 } {k}  -\frac {1} {\alpha } + \sigma }=
C (-t) ^{- 1 + (J +1 ) (1 - \frac {2} {k})  -\frac {1} {\alpha } + \sigma } .
\end{equation*}  
Note that by~\eqref{feApprox1:6} and~\eqref{feApprox1:7}, 
\begin{equation*} 
(J +1 ) (1 - \frac {2} {k})  -\frac {1} {\alpha } + \sigma \ge \frac {1} {2}(J +1 ) -\frac {1} {\alpha } + \sigma \ge 3 \sigma ,
\end{equation*} 
so that
\begin{equation} \label{fCite2} 
\| \Err _J\| _{ L^2 }  \|\varepsilon \| _{ L^2 } \le  C (-t) ^{- 1 + 3 \sigma } .
\end{equation}  
It follows from the above inequalities that
\begin{equation*} 
\frac {d} {dt}  \|\varepsilon (t)\| _{ L^2 }^2 \ge  - 2^{\alpha +1}\alpha ^{-1} |\lambda | \Csu  (-t)^{-1}  \| \varepsilon \| _{ L^2 }^2  - C (-t)^{ -1 + 2\sigma + \nu } ,
\end{equation*} 
where $\nu =  \min\{ 1, \frac {1} {\alpha }, \sigma  \}$; 
and so
\begin{equation} \label{fCite1} 
\begin{split} 
\frac {d} {dt}  \Bigl( (-t)^{- \sigma }  & \|\varepsilon (t)\| _{ L^2 }^2 \Bigr) =  \sigma  (-t)^{-\sigma -1} \|\varepsilon (t)\| _{ L^2 }^2 +  (-t)^{-\sigma } \frac {d} {dt}  \|\varepsilon (t)\| _{ L^2 }^2 \\ \ge  & [\sigma - 2^{\alpha +1}\alpha ^{-1} |\lambda | \Csu ] (-t)^{-\sigma -1} \|\varepsilon (t)\| _{ L^2 }^2 - C (-t)^{ -1 +  \sigma  + \nu}  .
\end{split} 
\end{equation} 
Using~\eqref{feApprox1:3}, we obtain 
\begin{equation*} 
\frac {d} {dt}  \Bigl( (-t)^{-\sigma }   \|\varepsilon (t)\| _{ L^2 }^2 \Bigr)   \ge   - C (-t)^{ -1 +  \sigma + \nu }  .
\end{equation*} 
Integrating on $(t, T_n)$ and using $\varepsilon (T_n) =0$, we deduce that 
\begin{equation*} 
 (-t)^{-\sigma }   \|\varepsilon (t)\| _{ L^2 }^2  \le  C (-t)^{ \sigma + \nu } ,
\end{equation*}
hence
\begin{equation} \label{fAP5} 
\|\varepsilon (t)\| _{ L^2 } \le C (-t)^{\sigma + \frac {\nu } {2}} , 
\end{equation} 
for all $t\in (\tau _n, T_n )$. 

We now define the energy 
\begin{equation} \label{fAPE3} 
E(t) = \frac {1} {2} \int    |\nabla \varepsilon (t)| ^2 - \frac {\IL } {\alpha +2} \int    | \varepsilon (t)|^{\alpha +2} .
\end{equation} 
Multiplying equation~\eqref{fAP2} by $-\Delta  \overline{\varepsilon } - \IL  \overline{f(\varepsilon )} $ and taking the real part, we obtain after integrating by parts
\begin{equation} \label{fAPE4} 
\begin{split} 
\frac {d} {dt} E(t) = &  \Re \int   \nabla  f(\varepsilon )  \cdot \nabla  \overline{\varepsilon }  + \Re \int   \lambda \nabla ( f (U_J+ \varepsilon ) - f(U_J) - f(\varepsilon ) ) \cdot \nabla  \overline{\varepsilon }   \\
& - \IL  \Re \int    \lambda ( f (U_J+ \varepsilon ) - f(U_J) )  \overline{f(\varepsilon )} + \Re \int   \nabla \Err_J \cdot \nabla  \overline{\varepsilon } \\ & - \IL  \Re \int    \Err_J \overline{f(\varepsilon )}
= : A_1 + A_2 + A_3 + A_4 + A_5 .
\end{split} 
\end{equation} 
Using~\eqref{fDerf2:b1}, we have
\begin{equation} \label{fAPE5} 
A_1 \ge \int  |\varepsilon |^\alpha  |\nabla \varepsilon |^2. 
\end{equation} 
Moreover, it follows from~\eqref{feANL1:2}, \eqref{fEUGU1},   \eqref{fAP4}, \eqref{fDerf7}  that
\begin{equation*} 
\begin{split} 
 |A_3| & \le  |\lambda |^2 \Csu \int (  |U_J|^\alpha  |\varepsilon |^{\alpha +2} +  |\varepsilon |^{2\alpha +2} )  \\ & 
 \le  C  (-t)^{ - 1+ 2\sigma  } +  \Bigl(   \frac {1} {8}+ C (-t)^{ \frac {4\sigma } {N} } \Bigr) \int  |\varepsilon |^\alpha   | \nabla \varepsilon  |^2.
\end{split} 
\end{equation*} 
We let $ \widetilde{s} <0$ be defined by
\begin{equation} 
 C (-  \widetilde{s}  )^{ \frac {4\sigma } {N} } = \frac {1} {8},
\end{equation} 
and we deduce that
\begin{equation} \label{fAPE6} 
 |A_3| \le  C  (-t)^{ - 1+ 2\sigma  } +    \frac {1} {4}  \int  |\varepsilon |^\alpha   | \nabla \varepsilon  |^2,
\end{equation} 
for all $n\ge - \frac {1} { \widetilde{s} }$ and all $\tau _n < t\le  T_n$ such that $t\ge  \widetilde{s} $. 

Next by~\eqref{feEstUJ:2}, \eqref{feEstUz:2} and~\eqref{fAP3}, 
\begin{equation*} 
\begin{split} 
 |A_4| & \le  \| \nabla \Err_J \| _{ L^2 }  \| \nabla \varepsilon \| _{ L^2 } \le C (-t)^{J( 1 - \frac {2} {k} ) - \frac {3} {k}  } \|U_0\| _{ L^2 }  \|\nabla \varepsilon \| _{ L^2 } \\ &
\le   C (-t)^{J( 1 - \frac {2} {k} ) - \frac {3} {k} -\frac {1} {\alpha } + (1-\theta )\sigma }  .
\end{split} 
\end{equation*} 
Moreover, using~\eqref{feEstUJ:2}, \eqref{feEstUz:2} and~\eqref{fAP4}, 
\begin{equation*} 
 |A_5| \le C \| \Err_J \| _{ L^{\alpha +2} }  \| \varepsilon  \| _{ L^{\alpha +2} }^{\alpha +1}
 \le  C (-t)^{J( 1 - \frac {2} {k} ) - \frac {2} {k} - \frac {1} {\alpha } + \frac {\alpha +1} {\alpha +2} ( \alpha +2 - \frac {N\alpha } {2} \theta )\sigma}   .
\end{equation*} 
Using $\min \{ 1-\theta, \frac {\alpha +1} {\alpha +2} ( \alpha +2 - \frac {N\alpha } {2} \theta ) \}\ge 1 - (N\alpha +1)\theta \ge 0$ (by~\eqref{feApprox1:5}), we conclude that
\begin{equation*} 
 |A_4 + A_5| \le  C (-t)^{J( 1 - \frac {2} {k} ) - \frac {3} {k} - \frac {1} {\alpha }  } 
 =  C (-t)^{- 1 +( J +1) (1 - \frac {3} {k} ) - \frac {1} {\alpha } + \frac {1} {k} J }   .
\end{equation*} 
Note that by~\eqref{feApprox1:6} and~\eqref{feApprox1:7} 
\begin{equation*} 
( J +1)  \Bigl( 1 - \frac {3} {k}  \Bigr) - \frac {1} {\alpha } + \frac {1} {k} J \ge ( J +1)  \Bigl( 1 - \frac {3} {k}  \Bigr) - \frac {1} {\alpha } \ge \frac {1} {2}( J +1) - \frac {1} {\alpha } \ge 2\sigma ,
\end{equation*} 
so that
\begin{equation} \label{rfAPE7} 
 |A_4 + A_5| \le  C (-t)^{- 1 + 2\sigma   }   .
\end{equation} 
We now estimate $A_2$. We write
\begin{equation*} 
\begin{split} 
\nabla  ( f (U_J+ \varepsilon ) - f(U_J) - f(\varepsilon ) ) 
=  & (df(U_J + \varepsilon  ) - df (\varepsilon )) \nabla \varepsilon   \\ & + 
 (df(U_J + \varepsilon ) - df (U_J )) \nabla U_J  ,
\end{split} 
\end{equation*} 
so that 
\begin{equation} \label{rfAPE8} 
 |A_2| \le  |\lambda |  \Bigl( \int B_1 + \int B_2  \Bigr) 
\end{equation} 
with 
\begin{equation*} 
B_1 =   |df(U_J + \varepsilon ) - df (\varepsilon )| \,  |\nabla \varepsilon |^2, \quad B_2 =   |df(U_J + \varepsilon ) - df (U_J )| \,  |\nabla U_J| \,  |\nabla \varepsilon | . 
\end{equation*} 
It follows from~\eqref{fDerf4:1:b1}  and~\eqref{feEstUJ:4}  that 
\begin{equation*} 
  |df(U_J + \varepsilon ) - df (\varepsilon  )|  \le  2^\alpha \Csu  |U_0|^\alpha + 2\Csu \Calpha  |\varepsilon |^{\alpha -1}  |U_0| .
\end{equation*} 
If $\alpha >1$, then $  | \varepsilon |^{\alpha -1}  |U_0| \le (8 \Csu  |\lambda |)^{-1}  |\varepsilon |^\alpha +(8 \Csu  |\lambda |)^{\alpha -1}  |U_0 |^\alpha $, so that (recall $ |\lambda |\ge 1$)
\begin{equation*} 
\begin{split} 
  |df(U_J + \varepsilon ) - df (\varepsilon  )|  & \le ( 2^\alpha \Csu   + 2\Csu (8 \Csu  |\lambda |)^{\alpha -1} ) |U_0 |^\alpha + \frac {1} {4  |\lambda |}  |\varepsilon |^\alpha  \\ & \le (8 \Csu  |\lambda |)^{\alpha }  |U_0 |^\alpha + \frac {1} {4  |\lambda |}  |\varepsilon |^\alpha  .
\end{split} 
\end{equation*} 
Using~\eqref{fEUZ2}, we deduce that
\begin{equation} \label{rfAPE11} 
\int B_1 \le \alpha (8 \Csu  |\lambda |)^{\alpha } (-t)^{-1}  \| \nabla \varepsilon \| _{ L^2 }^2 + \frac {1} {4  |\lambda |} \int  |\varepsilon |^\alpha  |\nabla \varepsilon |^2 .
\end{equation} 
To estimate $B_2$, we consider separately the cases $\alpha \le 1$, $1< \alpha \le 3$, and $\alpha >3$.

Suppose first $\alpha \le 1$. 
Using~\eqref{fEUGU1}, we see that
\begin{equation} \label{fEUGU2} 
 |U_J|^{\alpha -1}  |\nabla U_J| \le C (-t)^{- \frac {1} {k}}  |U_0|^\alpha \le  (-t)^{-1 - \frac {1} {k}} .
\end{equation} 
Now if $ |v |\le \frac {1} {2}  |u|$, then $ |u+sv| \ge \frac {1} {2}  |u|$ for all $0\le s\le 1$. Writing
 $df(u+v) - df(u) = \int _0^1 \frac {d} {ds} df(u+ sv)$, it follows easily that
\begin{equation*} 
 |df(u+v) - df(u) | \le C  \Bigl( \min _{ 0\le s\le 1 } | u + sv | \Bigr)^{\alpha -1}  |v| \le C  |u|^{\alpha -1} |v|,\quad  \text{if }  |v |\le \frac {1} {2}  |u| .
\end{equation*} 
It follows that
\begin{equation*} 
\begin{split} 
\int  _{  |\varepsilon  |\le  \frac {1} {2}  | U_J | } B_2 & \le C \int  |U_J| ^{\alpha -1}  |\varepsilon | \,  |\nabla U_J| \,  |\nabla \varepsilon |
\le C (-t)^{-1 - \frac {1} {k} }  \|\varepsilon \| _{ L^2 }  \| \nabla \varepsilon \| _{ L^2 } \\ & 
\le C (-t)^{-1 - \frac {1} {k} + (2-\theta )\sigma } 
\end{split} 
\end{equation*} 
where we used~\eqref{fEUGU2} and~\eqref{fAP3}.
Moreover, using~\eqref{feANL1:3} and~\eqref{fEUGU2},
\begin{equation*} 
\begin{split} 
  \int  _{  |\varepsilon  |> \frac {1} {2}  | U_J | } B_2 \le  & C   \int  _{  |\varepsilon  |> \frac {1} {2}  | U_J | }  |\varepsilon |^\alpha   |\nabla U_J| \,  |\nabla \varepsilon |   
\le  C   \int  _{  |\varepsilon  |> \frac {1} {2}  | U_J | }  |\varepsilon |^{\alpha  -1} |\nabla U_J| \, |\varepsilon |\,  |\nabla \varepsilon |   \\ \le & C   \int  |U_J |^{\alpha  -1} |\nabla U_J| \, |\varepsilon |\,  |\nabla \varepsilon |  \le  C (-t)^{-1 - \frac {1} {k} + (2-\theta )\sigma } .
\end{split} 
\end{equation*} 
Thus we see that
\begin{equation} \label{rfAPE12} 
\int   B_2  \le C (-t)^{-1 - \frac {1} {k} + (2-\theta )\sigma } .
\end{equation} 

When $\alpha >1$, we deduce from~\eqref{fDerf4:1:b1}, \eqref{fEUGU2} and~\eqref{fAP3} that 
\begin{equation*} 
\begin{split} 
\int B_2 & \le  \Csu \int   |U_J|^{\alpha -1} |\nabla U_J| \, |\varepsilon | \, |\nabla \varepsilon | + \Csu \int  |\nabla U_J| \, |\varepsilon |^\alpha  |\nabla \varepsilon |  \\ & \le C (-t)^{-1 - \frac {1} {k} + (2-\theta )\sigma } +  \Csu \int  |\nabla U_J| \, |\varepsilon |^\alpha  |\nabla \varepsilon | .
\end{split} 
\end{equation*} 

Suppose $1< \alpha \le 3$. By~\eqref{fDFG0} we have $0\le  N- (N-2)\alpha < 2\alpha $, and by Gagliardo-Nirenberg's inequality
\begin{equation*} 
 \| \varepsilon  \| _{ L^{2\alpha } } \le C  \| \nabla \varepsilon \| _{ L^2 }^{\frac {N(\alpha -1)} {2\alpha }}  \|\varepsilon \| _{ L^2 }^{\frac {N- (N-2) \alpha } {2\alpha }} .
\end{equation*} 
Using~\eqref{fEUGU1} and~\eqref{fAP3},
\begin{equation*} 
\begin{split} 
\int  |\nabla U_J| \, & |\varepsilon |^\alpha  |\nabla \varepsilon |  \le C  (-t)^{-\frac {1} {\alpha } - \frac {1} {k}}  \| \varepsilon \| _{ L^{2\alpha } }^\alpha  \| \nabla \varepsilon \| _{ L^2 } \le C  (-t)^{-\frac {1} {\alpha } - \frac {1} {k} + (\alpha +1 - \frac {N\alpha +2-N} {2} \theta )\sigma }\\ & \le C  (-t)^{- 1 - \frac {1} {k} + (\alpha +1 - \frac {N\alpha +2-N} {2} \theta )\sigma }
 = C  (-t)^{- 1 - \frac {1} {k}+ (2- \theta ) \sigma  + (\alpha -1) (1 - \frac {N} {2} \theta )\sigma } .
\end{split} 
\end{equation*} 
Using~\eqref{feApprox1:3}, we conclude that in this case 
\begin{equation} \label{rfAPE13} 
\int B_2 \le C (-t)^{-1 - \frac {1} {k} + (2-\theta )\sigma }  .
\end{equation} 

If $\alpha >3$, then by~\eqref{fEUGU1},
\begin{equation*} 
\begin{split} 
\Csu  |\nabla U_J| \, |\varepsilon |^\alpha  |\nabla \varepsilon | & \le \frac {1} {4  |\lambda |}  |\varepsilon |^\alpha  |\nabla \varepsilon |^2 + C |\varepsilon |^\alpha  |\nabla U_J |^2 \\ &
\le \frac {1} {4  |\lambda |}  |\varepsilon |^\alpha  |\nabla \varepsilon |^2 + C (-t)^{- \frac {2} {k}}  |\varepsilon |^\alpha  |U_0 |^2.
\end{split} 
\end{equation*} 
Applying~\eqref{fAP4} and~\eqref{feEstUz:2}, we obtain
\begin{equation*} 
\begin{split} 
\int  |\varepsilon |^\alpha  |U_0 |^2 & \le  \| \varepsilon \| _{ L^{\alpha +2} }^\alpha   \|U_0\| _{ L^{\alpha +2} }^2 
\le C (-t)^{ - \frac {2} {\alpha } + ( \alpha  - \frac {N\alpha ^2} {2 (\alpha +2)} \theta )\sigma } \\
& \le C (-t)^{ - 1 + ( \alpha  - \frac {N\alpha ^2} {2 (\alpha +2)} \theta )\sigma } .
\end{split} 
\end{equation*} 
Since $\alpha \ge 3$ and $N\le 3$,
\begin{equation*} 
\begin{split} 
 ( \alpha  - \frac {N\alpha ^2} {2 (\alpha +2)} \theta )\sigma & = (2- \theta )\sigma +   \Bigl( \alpha -2 - \frac {N\alpha ^2 } {2 (\alpha +2)} \theta + \theta  \Bigr) \sigma \\
& \ge  (2- \theta )\sigma +   ( 1 -  3\alpha ^2  \theta  ) \sigma 
\end{split} 
\end{equation*} 
Using~\eqref{feApprox1:5}, we deduce that
\begin{equation*} 
\int  |\varepsilon |^\alpha  |U_0 |^2 \le C (-t)^{ - 1 +  (2- \theta )\sigma  } ;
\end{equation*} 
and so,
\begin{equation*} 
  \Csu \int  |\nabla U_J| \, |\varepsilon |^\alpha  |\nabla \varepsilon | \le \frac {1} {4  |\lambda |} \int  |\varepsilon |^\alpha  |\nabla \varepsilon |^2 + C  (-t)^{ - 1 - \frac {2} {k} + (2- \theta )\sigma   } ,
\end{equation*} 
so that in this case
\begin{equation} \label{rfAPE14} 
\int B_2 \le  \frac {1} {4  |\lambda |} \int  |\varepsilon |^\alpha  |\nabla \varepsilon |^2 + C (-t)^{-1 - \frac {2} {k} + (2-\theta )\sigma }.
\end{equation} 
Estimates~\eqref{rfAPE8}, \eqref{rfAPE11}, \eqref{rfAPE12}, \eqref{rfAPE13}  and~\eqref{rfAPE14}  imply
\begin{equation*} 
 | A_2 | \le  \alpha  |\lambda | (8 \Csu  |\lambda |)^{\alpha } (-t) ^{-1} \| \nabla \varepsilon \| _{ L^2 }^2 +  \frac {1} {2}  \int  |\varepsilon |^\alpha  |\nabla \varepsilon |^2 + C (-t)^{-1 - \frac {2} {k} + (2-\theta )\sigma } .
\end{equation*} 
Using~\eqref{feApprox1:7}, we see that $- \frac {2} {k} + (2-\theta )\sigma \ge 2(1-\theta ) + \frac {\theta \sigma } {2}$, hence
\begin{equation}  \label{rfAPE15} 
 | A_2 | \le  \alpha  |\lambda | (8 \Csu  |\lambda |)^{\alpha } (-t) ^{-1} \| \nabla \varepsilon \| _{ L^2 }^2 +  \frac {1} {2}  \int  |\varepsilon |^\alpha  |\nabla \varepsilon |^2 + C (-t)^{-1 + 2(1-\theta )\sigma + \frac {\theta \sigma } {2} } .
\end{equation} 
Combining~\eqref{fAPE4}, \eqref{fAPE5}, \eqref{fAPE6}, \eqref{rfAPE7}  and~\eqref{rfAPE15}, we obtain
\begin{equation*} 
\frac {d} {dt} E(t) \ge  \frac {1} {4} \int  |\varepsilon |^\alpha  |\nabla \varepsilon |^2 -  \alpha  |\lambda | (8 \Csu  |\lambda |)^{\alpha } (-t)^{-1}  \| \nabla \varepsilon \| _{ L^2 }^2 - C (-t)^{-1 + 2(1-\theta )\sigma + \frac {\theta \sigma } {2} }  .
\end{equation*} 
Using~\eqref{fAP4} and~\eqref{feApprox1:3}, we deduce that 
\begin{equation}  \label{rfAPE17} 
\begin{split} 
\frac {d} {dt}[ (-t)^{- \sigma }  E(t) ] = & \sigma  (-t)^{- \sigma -1} E(t) + (-t)^{- \sigma } \frac {d} {dt} E(t )
\\ \ge &  \frac {1} {4} (-t)^{-\sigma  } \int  |\varepsilon |^\alpha  |\nabla \varepsilon |^2
 - C (-t)^{   -1 + (1- 2\theta )\sigma + \frac {\theta \sigma } {2} } .
\end{split} 
\end{equation} 
It follows from~\eqref{feApprox1:5}  that $(1-2 \theta )\sigma \ge 0$, so that  the power of $-t$  on  the right-hand side of the above inequality are (strictly) larger than $-1$.
 Integrating on $(t, T_n)$, using $\varepsilon (T_n)=0$, and multiplying by $(-t)^\sigma $, we obtain
\begin{equation*} 
\frac {1} {4} (-t)^\sigma  \int _t^{T_n }  (-s)^{-\sigma } \int  |\varepsilon |^\alpha  |\nabla \varepsilon |^2 +  E(t) \le  
C   (-t)^{  2(1-  \theta )\sigma + \frac {\theta \sigma } {2} }  .
\end{equation*} 
Using~\eqref{fAP4}, we deduce that 
\begin{equation} \label{fAP5:b1} 
 (-t)^\sigma  \int _t^{T_n }  \int  |\varepsilon |^\alpha  |\nabla \varepsilon |^2 +    \| \nabla \varepsilon \| _{ L^2 }^2  \le  
C (-t)^{ 2(1-\theta )\sigma + \frac {\theta \sigma } {2}  } 
\end{equation} 
for all $n\ge - \frac {1} { \widetilde{s} }$ and all $\tau _n < t\le  T_n$ such that $t\ge  \widetilde{s} $. 

We now conclude as follows. By~\eqref{fAP5} and~\eqref{fAP5:b1} (and since $2(1-\theta ) \sigma \ge  \sigma $), there exists $S<0$ such that for $n$ sufficiently large (so that $S< T_n$),
\begin{equation} \label{fAP5:b2} 
 \| \varepsilon (t)  \| _{ L^2 } \le   (-t)^\sigma , \quad   \| \nabla \varepsilon \| _{ L^2 }^2  \le   (-t)^{2 (1- \theta ) \sigma }, \quad \int  _t ^{T_n} \int  |\varepsilon |^\alpha  | \nabla \varepsilon |^2 \le 1 ,
\end{equation} 
for all $\tau _n <t < T_n$ such that $t\ge S$. 
By the definition~\eqref{fAP3}  of $\tau _n$, this implies $  \tau _n \le S$. 
 Using property~\eqref{fAP3:b1}, we conclude that $s_n < S$ and that~\eqref{feApprox1:1}, \eqref{feApprox1:2} and~\eqref{feApprox1:9} hold.  
\end{proof} 

\section{Proof of Theorem~$\ref{eThm1}$} \label{sProof} 

Let $ \Cset $ be any compact set of $\R^N$ included in the ball of center $0$ and radius~$1$
(by the scaling invariance of equation \eqref{NLS1}, this assumption does not restrict the generality).
It is well-known that there exists a smooth function $Z:\R^N\to [0,\infty)$ which vanishes exactly on $ \Cset $
(see \emph{e.g.}~Lemma 1.4, page 20 of \cite{MoRe}).
For $\alpha $ satisfying \eqref{fDFG0}, let $\sigma ,\theta ,J,k$ be defined by~\eqref{feApprox1:3}, \eqref{feApprox1:5}, \eqref{feApprox1:6} and~\eqref{feApprox1:7}.  
 Define the function $A:\R^N\to [0,\infty)$ by
\begin{equation} \label{fDfnA} 
A(x)=\left(Z(x) \chi(  |x|) + (1-\chi(  |x| )) |x|\right)^k,
\end{equation} 
where 
\begin{equation*} 
\begin{cases} 
 \chi \in C^\infty (\R, \R) \\
 \chi (s) = 
\begin{cases} 
1 & 0\le s\le 1 \\
0 & s\ge 2
\end{cases} 
\\
  \chi '(s) \le 0 \le \chi (s) \le 1,\quad s\ge 0.
\end{cases} 
\end{equation*} 
It follows that the function $A$ satisfies \eqref{on:A} and vanishes exactly on $ \Cset $.

We consider the solution $u_n$ of equation~\eqref{fAP1}, $\varepsilon _n$ defined by~\eqref{fDfnen}, and $n_0\ge 1$ and $S <0$ given by Proposition~\ref{eApprox1}. 
Using the estimate~\eqref{feANL1:2}  and the 
 embeddings $H^1 (\R^N ) \hookrightarrow L^{\alpha +2} (\R^N ) $, $L^{\frac {\alpha +2} {\alpha +1}} ( \R^N ) \hookrightarrow H^{-1} ( \R^N ) $, we deduce from equation~\eqref{fAP2} that
\begin{equation*} 
 \| \partial _t \varepsilon _n \| _{ H^{-1} } \lesssim \| \varepsilon _n  \| _{ H^1 } + \| U_J \| _{ H^1 }^\alpha \| \varepsilon _n \| _{ H^1 } + \| \varepsilon _n \| _{ H^1 }^{\alpha +1} + \|  \Err _J \| _{ L^2 }
\end{equation*} 
so that, applying~\eqref{feApprox1:1}, \eqref{feApprox1:2}, \eqref{feEstUJ:2}, \eqref{feEstUJ:3},  \eqref{feEstUz:1:0} and~\eqref{feEstUz:2}, there exists $\kappa >0$ such that
\begin{equation} \label{fEEN19} 
 \| \partial _t \varepsilon _n \| _{ H^{-1} } \le C(- t) ^{ - \kappa } ,\quad S\le t\le T_n  .
\end{equation} 
Given $\tau \in (S, 0 )$, it follows from~\eqref{feApprox1:1}, \eqref{feApprox1:2} and~\eqref{fEEN19} that $\varepsilon _n$ is bounded in $L^\infty ((S , \tau  ), H^1 (\R^N ) ) \cap W^{1, \infty } ((S ,\tau  ), H^{-1} (\R^N ) )$. 
Therefore, after possibly extracting a subsequence, there exists $\varepsilon  \in L^\infty ((S ,\tau  ), H^1 (\R^N ) ) \cap W^{1, \infty } ((S ,\tau  ), H^{-1} (\R^N ) )$ such that
\begin{gather} 
\varepsilon _n \goto _{ n\to \infty } \varepsilon  \text{ in } L^\infty ((S, \tau  ), H^1 (\R^N ) ) \text{ weak$^\star$} 
\label{fEEN21} \\
\partial _t \varepsilon _n  \goto _{ n\to \infty } \partial _t \varepsilon  \text{ in } L^\infty ((S, \tau  ), H^{-1} (\R^N ) ) \text{ weak$^\star$} \label{fEEN22} \\
\varepsilon _n (t) \goto _{ n\to \infty } \varepsilon  (t) \text{ weakly in $H^1 (\R^N ) $, for all $S\le t\le \tau  $} \label{fEEN23} 
\end{gather} 
Since $\tau \in (S, 0 )$ is arbitrary, a standard argument of diagonal extraction shows that there exists $\varepsilon  \in L^\infty _\Loc ((s, 0 ), H^1 (\R^N ) ) \cap W^{1, \infty } _\Loc ((S, 0 ), H^{-1} (\R^N ) )$ such that (after extraction of a subsequence) \eqref{fEEN21}, \eqref{fEEN22} and~\eqref{fEEN23} hold for all $S < \tau <0 $. 
Moreover, \eqref{feApprox1:1}, \eqref{feApprox1:2} and~\eqref{fEEN23} imply that
\begin{equation} \label{fEEN24} 
 \| \varepsilon  (t) \| _{ L^2 } \le (- t)^\sigma , \quad  \| \nabla  \varepsilon  (t) \| _{ L^2 } \le  (-t) ^{ (1-\theta ) \sigma },
\end{equation} 
for $S\le t<0$, and \eqref{fEEN19} and~\eqref{fEEN22} imply that
\begin{equation} \label{fEEN24:b1} 
 \| \partial _t \varepsilon  \| _{ L^\infty ((S, \tau  ), H^{-1} )} \le C (-\tau )^{ -\kappa } 
\end{equation} 
for all $S< \tau <0 $. 
In addition, it follows easily from~\eqref{fAP2} and the convergence properties~\eqref{fEEN21}--\eqref{fEEN23} that 
\begin{equation} \label{NLS6} 
\partial _t \varepsilon = i\Delta \varepsilon  + \lambda ( f (U_J+ \varepsilon ) - f(U_J) ) + \Err_J 
\end{equation} 
in $L^\infty _\Loc ((S, 0 ), H^{-1} (\R^N ) )$. 
Therefore, setting 
\begin{equation} \label{fEEN25} 
 u (t)= U_J (t) + \varepsilon  (t) \quad  S\le t<0
\end{equation} 
we see that $u \in L^\infty _\Loc (( S,0 ), H^1 (\R^N ) ) \cap W^{1, \infty } _\Loc (( S, 0 ), H^{-1} (\R^N ) )$ and, using~\eqref{fRA6}, that 
\begin{equation} \label{fEEN26:b1} 
\partial _t u= i\Delta u + \lambda f(u)
\end{equation} 
in $L^\infty _\Loc ((S , 0 ), H^{-1} (\R^N ) )$. 
By local existence in $H^1 (\R^N ) $ and uniqueness in $L^\infty _t H^1 _x$, we conclude that 
$ u\in C( (S , 0), H^1 (\R^N ) ) \cap C^1 ( ( S , 0), H^{-1} (\R^N ) ) $.

We now prove properties~\eqref{eThm1:1}, \eqref{eThm1:2} and~\eqref{eThm1:3}. 
Let $\Omega $ be an open subset of $\R^N $ such that $ \overline{\Omega } \cap \Cset = \emptyset$.
It follows from~\eqref{fDfnA}  that $A>0$ on $\Omega $; and so there exists a constant $c>0$ such that $A(x) \ge c (1 + |x| ) ^k$ on $\Omega $. 
Using~\eqref{fEUZ2}, we deduce that $ | U_0| \le C (1+  |x| )^{- \frac {k} {\alpha }}$ on $\Omega $. 
Since $ (1+  |x| )^{- \frac {k} {\alpha }} \in  L^2 (\R^N ) $ by~\eqref{feApprox1:7}, we conclude, applying~\eqref{feEstUJ:3} and~\eqref{feEstUz:1:0}, that 
\begin{equation} \label{fUpEsUj} 
\limsup _{ t \uparrow 0 }  \| U_J (t) \| _{ H^1 (\Omega ) } <\infty . 
\end{equation}   
Property~\eqref{eThm1:3} follows, using~\eqref{fEEN24}.
Let now $x_0\in \Cset$ and $r>0$, and set $\omega =  \{  |x-x_0| <r \}$. Let $p\ge 2$ satisfy $(N-2) p \le 2N$, so that $H^1 ( \omega ) \hookrightarrow L^p (\omega  )$. 
It follows from~\eqref{feEstUJ:4}, \eqref{feEstUz:2}  and~\eqref{e:24}  that 
\begin{equation*} 
 (- t)^{-\frac 1{\alpha }+\frac N{pk}} \lesssim  \|U_J (t)\|_{L^p (\omega )} \lesssim  \|U_J (t)\|_{L^p (\R^N  )} \lesssim  (- t)^{-\frac 1{\alpha }} .
\end{equation*} 
Using~\eqref{fEEN24} and the embedding $H^1 ( \omega ) \hookrightarrow L^p (\omega  )$ we deduce that 
\begin{equation} \label{fEstLp:2} 
(- t)^{-\frac 1{\alpha }+\frac N{pk}} \lesssim  \| u (t)\|_{L^p (\omega )} \lesssim  \| u (t)\|_{L^p (\R^N  )}  \lesssim  (- t)^{-\frac 1{\alpha }} .
\end{equation} 
Property~\eqref{eThm1:1} follows by letting $p=2$.
Next, we prove that
\begin{equation} \label{fEstLp:3} 
\lim _{ t\uparrow 0 }  \| \nabla u (t) \| _{ L^2 (\R^N ) } = \infty .
\end{equation} 
If $N\ge 3$, this follows from~\eqref{fEstLp:2} with $p= \frac {2N} {N-2}$ and Sobolev's inequality
\begin{equation*} 
 \| u (t)  \| _{ L^{\frac {2N} {N-2}} (\R^N  ) } \lesssim  \| \nabla u (t) \| _{ L^2 (\R^N ) }.
\end{equation*} 
If $N= 1$, we apply~\eqref{fEstLp:2} with $p=\infty $ and obtain using Gagliardo-Nirenberg's inequality
\begin{equation*} 
(-t)^{- \frac {2} {\alpha }} \lesssim  \| u(t) \| _{ L^\infty (\R) }^2  \lesssim  \| \nabla u (t)\| _{ L^2 (\R) }  \| u(t) \| _{ L^2 (\R)  } 
 \lesssim (-t) ^{-\frac {1} {\alpha }}  \| \nabla u (t)\| _{ L^2 (\R) } ,
\end{equation*} 
so that $  \| \nabla u (t)\| _{ L^2 (\R) } \gtrsim (-t) ^{- \frac {1} {\alpha }} $.
If $N=2$, we apply~\eqref{fEstLp:2} and Gagliardo-Nirenberg to obtain
\begin{equation*} 
(- t)^{-\frac 1{\alpha }+\frac N{pk}}  \lesssim  \| u (t)\|_{L^p (\R^2  )}  \lesssim  \| \nabla u(t)\| _{ L^2 (\R^2) }^{\frac {p-2} {p}}  \|  u(t)\| _{ L^2 (\R^2) }^{\frac { 2} {p}} \lesssim (-t) ^{- \frac {2} {\alpha p}} \| \nabla u(t)\| _{ L^2 (\R^2) }^{\frac {p-2} {p}}.
\end{equation*} 
For $p> 2 + \frac {N\alpha } {2}$, we deduce that $ \| \nabla u(t)\| _{ L^2 (\R^2) } \gtrsim (-t) ^{-\nu}$ with $\nu >0$.
This completes the proof of~\eqref{fEstLp:3}. 
Property~\eqref{eThm1:2} is an immediate consequence of~\eqref{fEstLp:3} and~\eqref{xnotE}.

The proof of Theorem~\ref{eThm1} is now complete. 

\begin{rem} \label{eCRem1} 
As observed at the beginning of Section~\ref{sBUAnsatz}, the construction of the blow-up ansatz does not require any upper bound on the power $\alpha $.
Theorem~\ref{eThm1} is restricted to $H^1$-subcritical powers because the energy estimates of Section~\ref{sApprox} only provide $H^1$ bounds. 
It is not too difficult to see that a similar result holds in the $H^1$-critical case $N\ge 3$ and $\alpha = \frac {4} {N-2}$. 
Indeed, in this case, the blow-up alternative is not that $ \| u(t)\| _{ H^1 }$ blows up, but that certain Strichartz norms blow up, for instance 
$ \| u  \| _{ L^{\frac {2N} {N-2}} _t  L^{\frac {2N^2} {(N-2)^2}} _x  }$. 
Control of this norm is given by estimate~\eqref{fAP5:b2} and the 
inequality
\begin{equation*} 
 \| u \| _{ L^{ \frac { 2 N^2 } { (N-2)^2 } } }^{ \frac {2N} {N-2} } =  \| \,  |u|^{\frac {\alpha +2} {2}} \| _{ L^{\frac {2N} {N-2}} }^2 \lesssim  \| \nabla (  |u|^{\frac {\alpha +2} {2}} ) \| _{ L^2 }^2 \lesssim \int  |u|^\alpha   | \nabla u |^2 . 
\end{equation*} 
For $H^1$ supercritical powers, higher order estimates would be required. 
It is not unlikely that a result similar to Theorem~\ref{eThm1} can be proved in the $H^2$-subcritical case $(N-4) \alpha <4$, by establishing $H^2$ estimates through $L^2$ estimates of $\partial _tu$, in the spirit of~\cite{Kato2}.
\end{rem}

\end{document}